\documentclass{amsart}

\usepackage[english]{babel}
\usepackage[utf8x]{inputenc}
\usepackage[T1]{fontenc}
\usepackage{comment}

\usepackage{amsmath}
\usepackage{enumitem}  
\usepackage{amssymb}
\usepackage{amsthm}
\usepackage{graphicx}
\usepackage[colorinlistoftodos]{todonotes}
\usepackage[colorlinks=true, allcolors=blue]{hyperref}
\usepackage{url}

\newcommand{\RNum}[1]{\uppercase\expandafter{\romannumeral #1\relax}}

\newcommand{\NN}{\mathbb{N}}
\newcommand{\ZZ}{\mathbb{Z}}
\newcommand{\RR}{\mathbb{R}}
\newcommand{\QQ}{\mathbb{Q}}
\newcommand{\TT}{\mathcal{T}}

\newcommand{\PP}{\mathbb{P}}

\newtheorem{thm}{Theorem}

\theoremstyle{definition}
\newtheorem{theorem}[thm]{Theorem}
\newtheorem{lemma}[thm]{Lemma}
\newtheorem{definition}[thm]{Definition}

\newtheorem{remark}[thm]{Remark}

\newtheorem{corollary}[thm]{Corollary}

\numberwithin{thm}{section}

\title[Counting simultaneous core partitions]
{Johnson's bijections and their application to counting simultaneous core partitions}

\author{Jineon Baek}
\address{Jineon Baek, University of Michigan, Department of Mathematics,  
2074 East Hall,
530 Church Street,
Ann Arbor, MI 48109-1043}
\email{jineon@umich.edu}

\author{Hayan Nam}
\address{Hayan Nam, University of California, Irvine, Department of
Mathematics, 340 Rowland Hall, Irvine, CA 92697}
\email{hayann@uci.edu}

\author{Myungjun Yu}
\address{Myungjun Yu, University of Michigan, Department of Mathematics,  
2074 East Hall,
530 Church Street,
Ann Arbor, MI 48109-1043}
\email{myungjuy@umich.edu}

\begin{document}

\begin{abstract}
Johnson recently proved Armstrong's conjecture which states that the average size of an $(a,b)$-core partition is $(a+b+1)(a-1)(b-1)/24$. He used various coordinate changes and one-to-one correspondences that are useful for counting problems about simultaneous core partitions. We give an expression for the number of $(b_1,b_2,\cdots, b_n)$-core partitions where $\{b_1,b_2,\cdots,b_n\}$ contains at least one pair of relatively prime numbers. We also evaluate the largest size of a self-conjugate $(s,s+1,s+2)$-core partition. 
\end{abstract}

\maketitle
\sloppy

\section{Introduction}
Let $\NN$ denote the set of non-negative integers. If $\lambda = (\lambda_1, \lambda_2, \cdots, \lambda_{\ell})$ is an $\ell$-tuple of non-increasing positive integers with $\sum_{i=1}^\ell \lambda_i = n$, then we call $\lambda$ a \emph{partition} of $n$. One can visualize $\lambda$ by using \emph{Ferrers diagram} as in Figure \ref{picture1}. Each square in a Ferrers diagram is called a \emph{cell}. By counting the number of cells in its NE (North East) and NW (North West) direction including itself, we define the \emph{hook length} of a cell. For example, the hook length of the colored cell in Figure \ref{picture1} is $6$.

\vspace{10mm}

\begin{figure}[h]
\includegraphics[width=5cm]{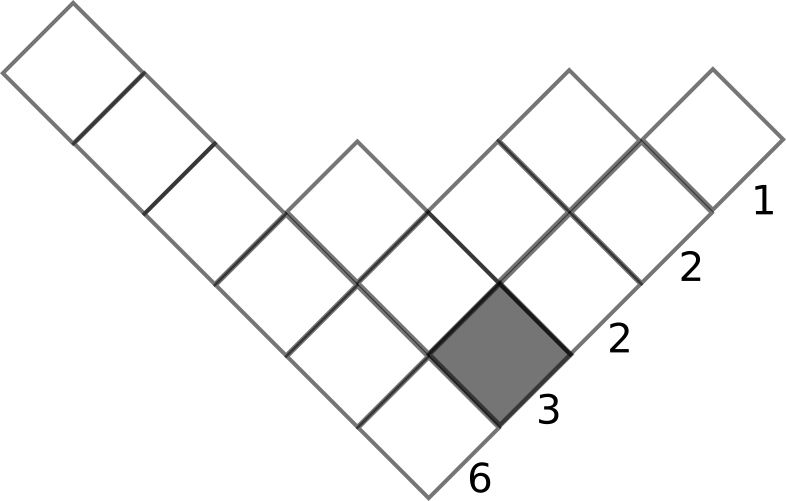}
\caption{$\lambda =(6,3,2,2,1)$}
\label{picture1}
\end{figure}

We say $\lambda$ is an \emph{$a$-core partition} (or, simply an \emph{$a$-core}) if there is no cell whose hook length is divisible by $a$. Similarly, we say a partition is an \emph{$(a_1, a_2, \cdots, a_n)$-core} if it is simultaneously an $a_1$-core, an $a_2$-core, $\cdots$, and an $a_n$-core.

Anderson \cite{And02} proved that if $a$ and $b$ are coprime, the number of $(a,b)$-cores is $\mathrm{Cat}_{a,b} := \frac{1}{a+b} {{a+b}\choose{a}}$, which is a generalized Catalan number.  Since Anderson \cite{And02}, many mathematicians have been conducting research on counting simultaneous core partitions and related subjects: \cite{article1}, \cite{article3}, \cite{Amd16}, \cite{Armstrong},  \cite{Johnson}, \cite{article8}, \cite{article9}, \cite{Str1}, \cite{Xionggenerating}, \cite{Xiong}, \cite{YQJZ}, \cite{article5}.

Armstrong \cite{Armstrong} conjectured that if $a$ and $b$ are coprime, the average size of an $(a,b)$-core partition is $(a+b+1)(a-1)(b-1)/24$. Johnson \cite{Johnson} recently proved Armstrong's conjecture by using Ehrhart theory. A proof without Ehrhart theory was given by Wang \cite{Wang}.

 In \cite{Johnson}, Johnson estabilished a bijection between the set of $(a,b)$-cores and the set 
$$
\left\{(z_0, z_1, \cdots , z_{a-1}) \in \NN^a : \sum_{i=0}^{a-1} z_i = b \text{ and } a \mid \sum_{i=0}^{a-1} iz_i\right\}.
$$
By showing that the cardinality of this set is $\mathrm{Cat}_{a,b}$, he gave a new proof of Anderson's theorem. Inspired by Johnson's method and this bijection, we count the number of simultaneous core partitions. We find a general expression for the number of $(b_1,b_2,…, b_n)$-core partitions where $\{b_1,b_2,…,b_n\}$ contains at least one pair of relatively prime numbers. As a corollary, we obtain an alternative proof for the number of $(s,s+d,s+2d)$-core partitions, which was given by Yang-Zhong-Zhou \cite{YZZ} and Wang \cite{Wang}. Subsequently, we also give a formula for the number of $(s,s+d,s+2d,s+3d)$-core partitions.

Many authors have studied core partitions satisfying additional restrictions.  For example, Berg and Vazirani \cite{Berg} gave a formula for the number of $a$-core partitions with largest part $x$. We generalize this formula, giving a formula for the number of $a$-core partitions with largest part $x$ and second largest part $y$.

This paper also includes a result related to the largest size of a simultaneous core partition which has been studied by many mathematicians. For example, Aukerman, Kane and Sze \cite[Conjecture 8.1]{AKL} conjectured that if $a$ and $b$ are coprime, the largest size of an $(a,b)$-core partition is $(a^2-1)(b^2-1)/24$. This was proved by Tripathi in \cite{Tripathi}. It is natural to wonder what would be the largest size of an $(a,b,c)$-core. Yang-Zhong-Zhou \cite{YZZ} found a formula for the largest size of an $(s,s+1,s+2)$-core. In section 4, we give a formula for the largest size of a self-conjugate $(s,s+1,s+2)$-core partition. We also prove that such a partition is unique (see Theorem \ref{largestself}).

The layout of this paper is as follows. In Section 2, we introduce Johnson's $c$-coordinates and $x$-coordinates for core partitions. In Section 3, we give a formula for the largest size of a self-conjugate $(s, s+1, s+2)$ core partition. In Section 4, using $c$-coordinates, we count the number of $a$-core partitions with given largest part and second largest part. In Section 5, we derive formulas for the number of simultaneous core partitions by using Johnson's $z$-coordinates.

\section{Review of Johnson's bijections}
In this section, we review Johnson's bijections in \cite{Johnson}, which are fundamental in this paper. For an integer $a$ greater than $1$, let $\PP_a$ denote the set of $a$-core partitions. Let 
$$
C_a := \left\{(c_0, c_1, \cdots, c_{a-1}) \in \ZZ^a: \sum_{i=0}^{a-1} c_i = 0\right\}.
$$
We first construct a bijective map from $C_a$ to $\PP_a$. 

\subsection{One-to-one correspondence between $\PP_a$ and $C_a$}
 
For each element $(c_0, c_1, \cdots, c_{a-1}) \in C_a$, we associate a ``\emph{tilted $a$-abacus}'' to it.

First, draw a vertical line $L$. Consider an infinite row of beads spaced $a$ units apart along with $a-1$ similar rows of beads below it, with each row shifted one unit to the right of the row above it (see Figure \ref{picture2}). Each bead will be colored black or white. No white bead is allowed on the right side of a black bead in the same row. In $i\textsuperscript{th}$ row, $0 \leq i \leq a-1$, we denote the number of white beads to the right of $L$ by $r_i$, and the number of black beads to the left of $L$ by $\ell_i$. Let $c_i=r_i-\ell_i$. Then, we have a tilted $a$-abacus for $(c_0, c_1, \cdots, c_{a-1})$.

Now we construct the corresponding $a$-core partition which is given by a path that consists of NE and SE steps. In a tilted $a$-abacus, let each black bead represent a NE step, and each white bead represent a SE step. The condition $\sum c_i = 0$ implies that the number of NE steps to the left of $L$ equals the number of SE steps to the right of $L$. Black beads to the left of the right-most white bead correspond to parts of the partition. Now, ignoring what row the beads are in, each part is obtained by counting the total number of white beads anywhere to the right of the black bead. For example, $(1,2,0,-3) \in C_4$ corresponds to the $4$-core partition $(9,6,3,1,1,1)$ in Figure \ref{picture2}. The map from $C_a$ to $\PP_a$ defined in this way is bijective (see \cite{Johnson} for details). 

\begin{figure}[h]
\includegraphics[width=10cm]{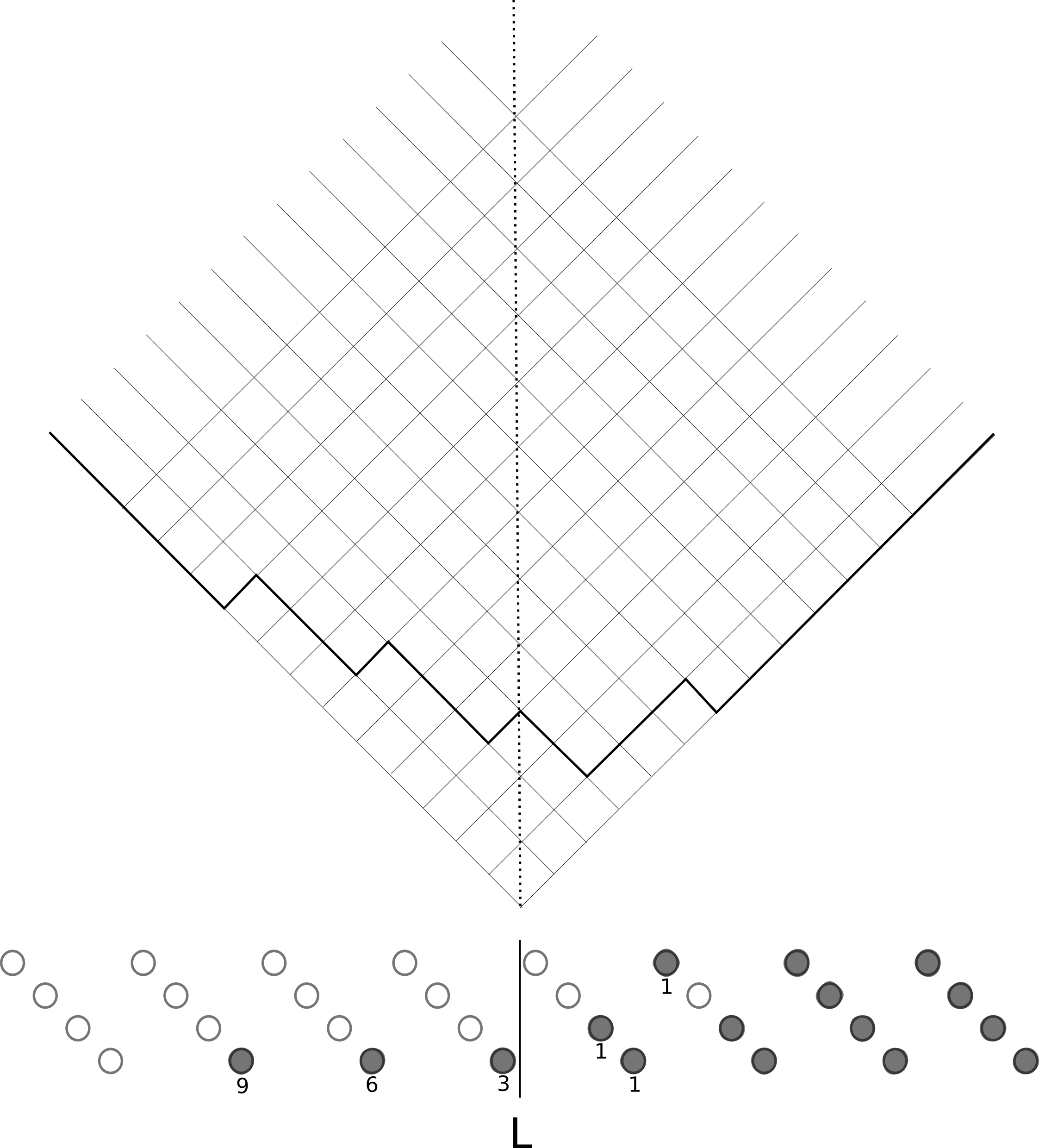}
\caption{$(1,2,0,-3) \in C_4$ and its associated $4$-core partition $(9,6,3,1,1,1)$. The part sizes below the black beads are the number of white beads anywhere to the right of that black bead.}
\label{picture2}
\end{figure}

We write $\varphi_a : \PP_a \to C_a$ for the inverse map. Define 
$$X_a:=\left\{(x_0, x_1, \cdots, x_{a-1}) \in \QQ^a : \sum_{i=0}^{a-1} x_i = 0 \text{ and } x_i \equiv \frac{i}{a} - \frac{a-1}{2a} \text{ (mod $1$)} \right\}. $$
There is a natural bijection from $C_a$ to $X_a$ by letting $x_i =  c_i + \frac{i}{a} - \frac{a-1}{2a}$. The composition of this map with $\varphi_a$ gives the bijection
$$
\psi_a : \PP_a \to X_a.
$$

For the rest of this section, we fix an $a$-core $\lambda$, $\varphi_a(\lambda) = (c_0, c_1, \cdots, c_{a-1})$, and $\psi_a(\lambda) = (x_0, x_1, \cdots, x_{a-1})$. Below are several lemmas which we use throughout the paper.

\begin{lemma}\cite[Theorem 2.10]{Johnson}
\label{csize}
The size of $\lambda$ is
$$
\sum_{k=0}^{a-1} \left(\frac{a}{2} c_k^2 +kc_k\right).
$$
\end{lemma}

\begin{lemma}
\label{selfconj}
For each $0 \le i \le a-1$, the partition $\lambda$ is self-conjugate if and only if $c_i = - c_{a-1-i}$. 
\end{lemma}

\begin{proof}
The lemma follows from the construction of the bijection beween $\PP_a$ and $C_a$. 
\end{proof}

\begin{lemma}\cite[Lemma 3.1]{Johnson}
\label{bcorecondc}
The partition $\lambda$ is also a $b$-core, and therefore an $(a,b)$-core, if and only if for any $0 \le i \le a-1$,
$$c_{i+b} - c_i \le q_a(b+i),$$ 
where $q_a(b+i)$ is the remainder (between $0$ and $a-1$) when $b+i$ is divided by $a$. 
\end{lemma}

\begin{lemma}\cite[Lemma 3.4]{Johnson}
\label{xlemma}
The partition $\lambda$ is also a $b$-core partition, and therefore an $(a,b)$-core, if and only if for any $0 \le i \le a-1$, 
$$x_{i+b} -x_i \le b/a.$$
\end{lemma}

\section{The largest size of a self-conjugate $(s,s+1,s+2)$-core partition}
The goal of this section is to give a formula for the largest size of a self-conjugate $(s, s+1, s+2)$-core partition. Yang, Zhong, and Zhou \cite{YZZ} evaluated the largest size of an $(s,s+1,s+2)$-core partition depending on the parity of $s$.

\begin{theorem}[Yang-Zhong-Zhou]
\label{yzz}
The largest size of an $(s, s+1, s+2)$-core partition is
\[
\begin{cases}
m {{m+1}\choose{3}} &\text{ if }  s=2m-1,\\
(m+1) {{m+1}\choose {3}} + {{m+2}\choose{3}} &\text{ if } s=2m.
\end{cases}
\]
Moreover, the largest size partition comes from a unique self-conjugate partition when $s$ is even, and a unique pair of conjugate partitions if $s$ is odd.
\end{theorem}

We devote this section to find the largest size of a \emph{self-conjugate} $(s,s+1,s+2)$-core partition.

\begin{remark}
Our result in Theorem \ref{largestself} coincides with Theorem \ref{yzz} if $s$ is even. In this case, we get a unique largest size partition. On the other hand, when $s$ is odd, the size of the largest self-conjugate $(s,s+1,s+2)$-core is smaller than the size of the largest unrestricted $(s,s+1,s+2)$-core, and the difference between them is 
\[
\begin{cases}
(2w-1)w^2 &\text{ if }  s=4w-1,\\
(2w-1)(w-1)^2 &\text{ if } s=4w-3.
\end{cases}
\]
If $s$ is odd, Yang, Zhong, and Zhou \cite{YZZ} showed there are two $(s,s+1,s+2)$-cores (a pair of conjugate partitions) with the largest size, whereas there is a unique self-conjugate $(s,s+1,s+2)$-core with the largest size (see Theorem \ref{largestself}). 
\end{remark}

Let $\lambda$ be a self-conjugate $(s, s+1, s+2)$-core partition. Let $a= s+1$ and $\varphi_a(\lambda) = (c_0, c_1, \cdots c_{s})$. Then Lemma \ref{bcorecondc} shows

\begin{equation}
\label{cccond}
\begin{split}
-1 &\le c_1- c_0 \le 1, \\
-1 &\le c_2- c_1 \le 1, \\
&\vdots \\
-1 &\le c_{s}- c_{s-1} \le 1,\\ 
0 & \le c_0- c_s \le 2.
\end{split}
\end{equation}

\begin{thm}
\label{largestself}
The largest size of a self-conjugate $(s, s+1, s+2)$-core partition is
\[
\begin{cases}
\frac{w(2w+1)(4w^2+2w+1)}{3}& \text{ if } s= 4w,         \\
\frac{w^2(8w^2-6w+1)}{3} &\text{ if } s= 4w-1,      \\
\frac{w(2w-1)(4w^2-2w+1)}{3} &\text{ if } s= 4w-2,     \\
\frac{(w-1)(2w-1)(4w^2-5w+3)}{3} & \text{ if } s= 4w-3.
\end{cases}
\] Moreover, there is a unique self-conjugate $(s,s+1,s+2)$-core partition having the largest size.
\end{thm}

\begin{proof}
Recall $a=s+1$. We give details for the cases $s = 4w-1$ and $s = 4w-3$.  The other cases can be proved similarly, or follow from Theorem \ref{yzz}. First assume that $s=4w-1$. Note that $c_{2w-1} =c_{2w} =0$ and $c_0 = 0$ or $1$ by Lemma \ref{selfconj} and \eqref{cccond}. Lemma \ref{csize} and Lemma \ref{selfconj} gives the size of a partition in terms of $c_i$-coordinates:
\begin{equation}
\label{size}
\sum_{k=0}^s \left(\frac{s+1}{2} c_k^2 +kc_k\right) = \sum_{k=0}^{2w-1} \left(4wc_k^2 - (4w-1- 2k)c_k\right). 
\end{equation}
For simplicity, let 
$$
f_k(c_k) = 4wc_k^2 - (4w-1-2k)c_k.
$$
For $0\le k \le 2w-1$, we define $m_k := \text{min}(k+1, 2w-1-k)$ and $n_k := -\text{min}(k, 2w-1-k)$. We claim \eqref{size} has its maximum either when
\begin{enumerate}[label=(\roman*)]
\item
$c_k = m_k$ for $0\le k \le 2w-1$, or
\item
$c_k = n_k$ for $0\le k \le 2w-1$.
\end{enumerate}

Note that case (i) is when $c_k$'s are ``as large as possible'' and case (ii) is when $c_k$'s are ``as small as possible'' under the restrictions on $c_i$. In case (i), $\{c_k\}$ is increasing for $0 \le k \le w-1$ and decreasing for $w-1 \le k \le 2w-1$ with the peak $c_{w-1} = w$. In case (ii), $\{c_k\}$ is decreasing for $0 \le k \le w-1$ and increasing for $w \le k \le 2w-1$ with two lowest terms $c_{w-1} = c_w = -(w-1)$.

Let $c_0 = 0$. It is clear that $f_k(c_k)$ has its maximum when $c_k = -\text{min}(k, 2w-1-k)$, so \eqref{size} gets its maximum in case (ii) . If $c_k > 0$ for $0 \le k \le 2w-2$, then $f_k(c_k)$ has its maximum if $c_k = \text{min}(k+1, 2w-1-k)$, which is the case (i). Finally, if $c_0  = 1$ and there is some $1 \le i \le 2w-2$ such that $c_i = 0$, we may assume $i$ is the smallest such index. By a similar reasoning as above we conclude that \eqref{size} acquires its maximum (in the case $c_i =0$) when 
\[c_k = l_k :=
\begin{cases}
\text{min}(k+1, i-k)    & \text{ if } 0 \le k \le i,         \\
-\text{min}(k-i, 2w-1-k)    &\text{ if }   i \le k \le 2w-1.   
\end{cases}
\] 
However, the maximum size in this case is bounded by the value of \eqref{size} in case(ii). To justify it, it is enough to note the following:
\begin{align*}
f_k(l_k) &< f_{k+1}(n_{k+1})   \text{ if } 0\le k \le i-1 \\
f_k(l_k) &= f_0(n_0) = 0 \text{ if } k = i \\
f_k(l_k)& \le f_{k}(n_k) \text{ if } i+1 \le k \le 2w-1
\end{align*}
By simple calculations, one can see
\begin{align*}
\sum_{k=0}^{2w-1} f_k(m_k)  &=  \frac{w^2(8w^2-6w+1)}{3}, \\
\sum_{k=0}^{2w-1} f_k(n_k) &=  \frac{w^2(8w^2-6w-2)}{3}. 
\end{align*}
Therefore, the maximum of \eqref{size} is $ \frac{w^2(8w^2-6w+1)}{3}$.

Similarly, if $s= 4w-3$ then $c_{2w-2} = c_{2w-1} = 0$, and the size is
\begin{equation}
\label{size2}
\sum_{k=0}^s \left(\frac{s+1}{2} c_k^2 +kc_k\right) = \sum_{k=0}^{2w-2} \left((4w-2)c_k^2 - (4w- 3 - 2k)c_k\right). 
\end{equation}
By the exactly same argument as above, the equation \eqref{size2} has its maximum either when
\begin{enumerate}[label=(\roman*)]
\item
$c_k = \text{min}(k+1, 2w-2-k)$ for $0\le k \le 2w-2$, or
\item
$c_k = -\text{min}(k, 2w-2-k)$ for $0\le k \le 2w-2$.
\end{enumerate}
Simple computations show \eqref{size2} has its maximum in the latter case, so the maximum of \eqref{size2} is $\frac{(w-1)(2w-1)(4w^2-5w+3)}{3}$. Note that it follows from the proof that if $\sum_{k=0}^s \left(\frac{s+1}{2} c_k^2 +kc_k\right)$ has its maximum value, then $c_k$ is determined uniquely, so there is a unique self-conjugate $(s,s+1,s+2)$-core of the largest size.  
\end{proof}


\section{Simultaneous Core Partitions with Fixed Largest Part}

 In this section, the convention is that $\binom{a}{b}=0$ if either $a$ or $b$ is negative.
We begin with a standard combinatorial fact.
\begin{lemma}
\label{numintsol}
 Let $b \in \ZZ$ and $s_i \in \ZZ$ for $0 \le i \le a-1$. Then
 $$
  \left|\left\{(z_0, z_1, \cdots, z_{a-1}) \in \ZZ^a : z_i \ge s_i \text{ and } \sum_{i=0}^{a-1} z_i = b\right\}\right| = {{b+a - \sum_{i=0}^{a-1} s_i -1}\choose{a-1}}.
 $$
\end{lemma}

\begin{lemma}
\label{partdiff}
If a partition $\lambda=(\lambda_1, \lambda_2, \cdots, \lambda_n)$ is an $a$-core, then 
\begin{enumerate}
\item
$\lambda_k - \lambda_{k+1} \le a-1$ for all $1 \le k \le n$.  
\item
For any $s \in \NN$, $s$ appears at most $a-1$ times among $\lambda_k$ for $1 \le k \le n$. 
\end{enumerate}
\end{lemma}

\begin{proof}
If $\lambda_i-\lambda_{i+1} \ge a$ for some $i$, it is easy to see that there is a cell with hook length $a$ on the $i\textsuperscript{th}$ row of the Ferrers diagram of $\lambda$, which proves (1). Now (2) follows from (1) by considering the conjugate of $\lambda$.   
\end{proof}

Berg and Vazirani proved the following.

\begin{theorem}[Berg, Vazirani]
The number of $\ell$-core partitions with largest part $k$ is $\binom{\ell+k-2}{k}$.
\end{theorem}

In the rest of the section we give a new proof and generalize the theorem.

\begin{thm} 
Let $i \ge 1$ be an integer. Let $\PP_a^{x,i}$ denote the set of $a$-core partitions $\lambda$ such that the largest part of $\lambda$ is $x$ and there are exactly $i$ parts of $\lambda$ equal to $x$. Then, we have the following.
\label{cmainthm}
\begin{enumerate}
\item
$|\PP_a^{x,i}| = {{x+a-2-i}\choose{a-1-i}}.$
\item
The number of $a$-core partitions with largest part $x$ is ${x+a-2} \choose {x}$. 
\item
The number of $a$-core partitions with largest part $x$ and second largest part $y$ is 
\[
\begin{cases}
{y+a-3} \choose {y} &\text{ if } x-y < a-1, \\ 
{{y+a-2}\choose{y}}  &\text{ if } x-y = a-1.
\end{cases}
\]
\end{enumerate}
\end{thm}

\begin{proof}
Recall the construction of $\varphi_a$ in subsection 2.1. Let $\lambda$ be a partition with $i$ many parts of $x$ where $x$ is the largest part. Let $\varphi_a(\lambda) = (c_0, c_1, \cdots, c_{a-1})$ such that $\sum_{k=0}^{a-1} c_k = 0$. Let $k_1: = \text{min}(t: c_t  \le c_s \text{ for } t \neq s )$. Then by the construction of $\varphi_a$, we have
$$
x = (c_0-c_{k_1}-1) + (c_1-c_{k_1}-1) + \cdots (c_{k_1-1} - c_{k_1} -1) + (c_{k_1+1}-c_{k_1}) + \cdots +(c_{a-1} -c_{k_1}).
$$
This implies $-ac_{k_1} - k_1 = x$, which determines $k_1$ since $-k_1 \equiv x$ (mod $a$), and $$c_{k_1}= -(x+k_1)/a.$$ First we assume $i +k_1 \le a-1$. Then we have
\[
\begin{cases}
c_j  \ge -\frac{x+k_1}{a} + 1 &\text{ if } 0 \le j \le k_1-1, \\ 
c_j = -\frac{x+k_1}{a} &\text{ if } k_1 \le j \le k_1+i-1, \\
c_j \ge -\frac{x+k_1}{a} + 1 &\text{ if } j = k_1 + i, \\
c_j  \ge -\frac{x+k_1}{a} &\text{ if } k_1 + i + 1 \le j \le a-1.
\end{cases}
\]
Then, by Lemma \ref{numintsol}, we have that $\sum_{k=0}^{a-1} c_k = 0$ has ${{x+a-2-i}\choose{a-1-i}}$ values of $(c_0, \cdots, c_{a-1})$ satisfying these conditions. Now we assume $i +k_1 \ge a$. Then we have
\[
\begin{cases}
c_j  = -\frac{x+k_1}{a} + 1 &\text{ if } 0 \le j \le i-(a-k_1)-1, \\ 
c_j \ge -\frac{x+k_1}{a} + 2 &\text{ if } j = i-(a-k_1), \\
c_j \ge -\frac{x+k_1}{a} + 1 &\text{ if } i-(a-k_1)+1 \le j \le k_1-1, \\
c_j  = -\frac{x+k_1}{a} &\text{ if } k_1 \le j \le a-1.
\end{cases}
\]
Again, it follows from Lemma \ref{numintsol} that $\sum_{k=0}^{a-1} c_k = 0$ has ${{x+a-2-i}\choose{a-1-i}}$ values of $(c_0, \cdots, c_{a-1})$ satisfying these conditions, and this complete the proof of (1). The second assertion follows from (1). 

Now suppose $\lambda$ has $y$ as the second largest part. Note that $x-y \le a-1$ by Lemma \ref{partdiff}. If $x-y = a-1$, then the second black bead from left locates on the same runner ($k_1\textsuperscript{th}$ runner) as the first black bead from left. This shows
\[
\begin{cases}
c_j \ge -\frac{x+k_1}{a} + 2 &\text{ if } 0 \le j \le k_1-1, \\ 
c_j  \ge -\frac{x+k_1}{a} + 1 &\text{ if } k_1 + 1 \le j \le a-1,
\end{cases}
\]
so there are ${{y+a-2}\choose{y}}$ values of $(c_0, \cdots, c_{a-1})$ satisfying these conditions due to Lemma \ref{numintsol}. If $x-y <a-1$, we first assume $k_2 := k_1 + x-y +1 \le a-1$. Here, $c_{k_2} = c_{k_1} = -\frac{x+k_1}{a} $ and
\[
\begin{cases}
c_j \ge -\frac{x+k_1}{a}+1 &\text{ if } 0 \le j \le k_1-1, \\
c_{j} \ge -\frac{x+k_1}{a}+1 &\text{ if } k_1 + 1 \le j \le k_2-1,\\
c_{j} \ge -\frac{x+k_1}{a} &\text{ if } k_2 + 1 \le j \le a-1.
\end{cases}
\]  
Then $\sum_{k=0}^{a-1} c_k = 0$ has ${{y+a-3}\choose{y}}$ values of $(c_0, \cdots, c_{a-1})$ satisfying these conditions by Lemma \ref{numintsol}. Now we assume $k_2 = k_1 + x-y +1 \ge a$. Then we have
$c_{k_2-a} = -\frac{x+k_1}{a} + 1$ and
\[
\begin{cases}
c_j \ge -\frac{x+k_1}{a}+2 &\text{ if } 0 \le j \le k_2-a-1, \\
c_{j} \ge -\frac{x+k_1}{a}+1 &\text{ if } k_2-a + 1 \le j \le k_1-1,\\
c_{j} \ge -\frac{x+k_1}{a} +1 &\text{ if } k_1 + 1 \le j \le a-1.
\end{cases}
\]  
Then $\sum_{k=0}^{a-1} c_k = 0$ has ${{y+a-3}\choose{y}}$ values of $(c_0, \cdots, c_{a-1})$ satisfying these conditions by Lemma \ref{numintsol} in this case too, so (3) follows.
\end{proof}


\section{Counting simultaneous core partitions}
We continue to assume that $a \ge 2$ is an integer. We also keep the convention that $\binom{a}{b}=0$ if any of $a$ or $b$ is negative.
\begin{definition}
Let $\TT$ denote the operator on $\NN^a$ such that  
$$
\TT(z_0, z_1, \cdots, z_{a-1}) = (z_1, \cdots, z_{a-1}, z_0),
$$
for any $(z_0, z_1, \cdots, z_{a-1}) \in \NN^a$. Let $S$ be a set of tuples $(z_0, z_1, \cdots z_{a-1}) \in \NN^a$. We say $S$ is \emph{stable} under $\TT$ if $\TT(S) = S$.
\end{definition}

\begin{lemma}
\label{mainlem}
Suppose $(a,b) = 1$. Suppose $S$ is stable under $\TT$.  Let
\begin{align*}
Y_1 &= \left\{(z_0, z_1, \cdots, z_{a-1}) \in \NN^a : \sum_{m=0}^{a-1} z_m = b, \text{ and } a \mid \sum_{m=0}^{a-1} mz_m \right\} \cap S,\\
Y_2 &= \left\{(z_0, z_1, \cdots, z_{a-1}) \in \NN^a : \sum_{m=0}^{a-1} z_m = b\right\} \cap S.
\end{align*}
Then $|Y_1| = |Y_2|/a$.
\end{lemma}

\begin{proof}
There is a natural action of $\ZZ/a\ZZ$ on $Y_2$ such that $\overline{1} \in \ZZ/a\ZZ$ acts as $\TT$. Let $f$ be the function on $Y_2$ such that $f((z_0, z_1, \cdots z_{a-1})) = \sum_{m=0}^{a-1} mz_m$. Let $s\in Y_2$. Note that
\begin{enumerate}
\item
$f(\TT^i(s))$ and $f(\TT^j(s))$ are in different residue classes modulo $a$ for $0 \le i \neq j \le a-1$, and
\item
Each orbit has $a$ elements.
\end{enumerate}
Noting that $a$ and $b$ are coprime proves (1), and (2) follows from (1) consequently. Now the lemma follows immediately. 
\end{proof}

\begin{remark}
As Aaron Berger pointed out, a more general version of Lemma \ref{mainlem} holds: Let  $\Theta : Y_2 \to \RR$ be a \emph{$\TT$-invariant} function, i.e, 
$\Theta \circ \TT = \Theta$. Then we have
$$
\sum_{x \in Y_1} \Theta(x) = \frac{1}{a} \sum_{y \in Y_2}\Theta(y).
$$
The proof is identical to that of Lemma \ref{mainlem}, and one can recover Lemma \ref{mainlem} simply by taking $\Theta = 1$. We will not need this general fact in the sequel.
\end{remark}

The following theorem gives an expression for the number of $(a,b_0, b_1, \cdots, b_n)$-core partitions for any non-negative integer $n$ when $a$ and $b_0$ are coprime. Note that when $n=0$, the following theorem is the same as Lemma 3.5 in \cite{Johnson}. In fact, Johnson used this to give a new proof of Anderson's theorem on the number of $(a,b)$-core partitions.

\begin{thm}
\label{mainthm}
Let $a, b_0, b_1, \cdots, b_n$ be positive integers, where none of $b_i$ is a multiple of $a$. Suppose that $a$ and $b_0$ are coprime. For $1 \le i \le n$, let $l_i$ be such that $1 \le l_i \le a-1$ and $a \mid b_0l_i+b_i$.  Then there is a bijection between the set of $(a, b_0,b_1, \cdots, b_n)$-core partitions and 
$$
\left\{(z_0, \cdots, z_{a-1}) \in \NN^a  :  \sum_{m=0}^{a-1}  z_m = b_0, a \mid \sum_{m=0}^{a-1} mz_m, \text{ and }
\sum_{m=j}^{j+l_i-1} z_m \le (b_0l_i + b_i)/a \right\}, 
$$
where the inequality holds for all $1 \le i \le n$ and $0 \le j \le a-1$. Here, indices are interpreted modulo $a$.
In particular, the number of $(a, b_0,b_1, \cdots, b_n)$-core partitions is
$$
\frac{1}{a}\left|\left\{(z_0, \cdots, z_{a-1}) \in \NN^a  :  \sum_{m=0}^{a-1} z_m = b_0, \text{ and } \sum_{m=j}^{j+l_i-1} z_m \le (b_0l_i + b_i)/a \text{ for all $i,j$ }  \right\}\right|.
$$
\end{thm}

\begin{proof}
We follow the proof of \cite[Lemma 3.5]{Johnson}. Lemma \ref{xlemma} implies that the set of $(a, b_0,b_1, \cdots, b_n)$-core partitions is in one-to-one correspondence with the set
$$
\{(x_0, x_1, \cdots, x_{a-1}) \in X_a: x_{i+b_j} - x_i \le b_j/a \text{ for all $i,j$ }\}.
$$
Let $k \equiv -(b_0 + 1)/2 \text{ (mod $a$)}$. Then $x_k = c_k + \frac{k}{a} - \frac{a-1}{2a}$, so $ax_k \equiv \frac{-a-b_0}{2}$ (mod $a$).  Let
\begin{equation}
\label{mainthmeq1}
z_m := x_{mb_0+k} - x_{(m+1)b_0 + k} + b_0/a. 
\end{equation}
Then $\sum_{m=0}^{a-1} z_m = b_0$ and $z_m \ge 0$  by Lemma \ref{xlemma}. Also, 
\begin{align*}
z_j + z_{j+1} + \cdots z_{j+l_i-1} & = x_{jb_0+k} - x_{jb_0+k + l_ib_0} +l_ib_0/a \\
&=  x_{jb_0+k + l_ib_0 + b_i}- x_{jb_0+k + l_ib_0} + l_ib_0/a \\
&\le (l_ib_0 + b_i)/a,  
\end{align*}
where the last inequality follows from Lemma \ref{xlemma} for $b_i$.  
Moreover, 
\begin{equation}
\label{mainthmeq2}
\sum_{m=0}^{a-1} mz_m = -ax_k +\frac{b_0(a-1)}{2} \equiv 0 \pmod{a}, 
\end{equation}
so $\sum_{m=0}^{a-1} mz_m$ is a multiple of $a$. Similarly, by \eqref{mainthmeq1} and \eqref{mainthmeq2} it is easy to see a tuple ($z_0, z_1, \cdots z_{a-1}$) satisfies the conditions
$$
\sum z_m =b_0, a|\sum mz_m, \text{ and } z_j + z_{j+1} + \cdots z_{j+l_i-1} \le (b_0l_i + b_i)/a \text{ for all $i,j$},
$$ 
determines $x_k$ uniquely, so does $x_j$ for any $j$.
Finally, Lemma \ref{mainlem} justifies the last assertion. 
\end{proof}

\begin{corollary}
\label{firstcor}
Suppose $(a,b_0) =1$ and $a \mid (b_0 + b_1)$. Let $m= (b_0 +b_1)/a$. Then the number of $(a,b_0,b_1)$-core partitions is
$$
\frac{1}{a} \sum_{my_m + \cdots 2y_2 + y_1 = b_0} \binom{a}{y_m} \binom{a-y_m}{y_{m-1}}\cdots \binom{a-y_m-\cdots - y_2}{y_1}.
$$
\end{corollary}

\begin{proof}
In the summand, the letter $y_i$ stands for the number of $i$'s in the tuple $(z_0, \cdots, z_{a-1})$. Then the corollary follows from the previous theorem. 
\end{proof}

Letting $a=s+1, b_0=s, b_1=s+2$, Corollary \ref{firstcor} recovers a theorem of Yang-Zhong-Zhou \cite{YZZ}. Letting $a= s+d, b_0 = s, b_1 = s+2d$, Corollary \ref{firstcor} recovers Theorem 1.6 of Wang \cite{Wang}. 

\begin{corollary}[Yang-Zhong-Zhou, Wang]
\label{YZZW}
Suppose $s$ and $d$ are coprime. Then the number of $(s, s+d, s+2d)$-core partitions is
$$
\frac{1}{s+d} \sum_{y_2=0}^{\lfloor s/2 \rfloor} {{s+d}\choose{y_2}}{{s+d-y_2}\choose{s-2y_2}}.
$$
\end{corollary}

\begin{thm}
Suppose $(s,d) = 1$. The number of $(s, s+d, s+2d, s+3d)$-core partitons is
$$
\frac{1}{s+d}\displaystyle\sum_{k=0}^{\lfloor s/2 \rfloor} \bigg\{{{s+d-k-1}\choose{k-1}}+{{s+d-k}\choose{k}}\bigg\}{{s+d-k}\choose{s-2k}}.
$$
\end{thm}

\begin{proof}
Putting $a= s+d, b_0=s, b_1= s+2d$, and $b_2= s+3d$ in Theorem \ref{mainthm}, it is enough to compute $\frac{1}{s+d}|A|$, where
$$ 
A := \left\{(z_0, \cdots, z_{s+d-1}) \in \NN^{s+d} : \sum_{i=0}^{s+d-1} z_i =s, 0 \le z_j \le 2, \text{ and } z_j + z_{j+1} \le 3 \text{ for all $j$}\right\}.
$$
Let a tuple $(z_0, z_1, \cdots, z_{s+d-1})$ be an element of $A$ and $k$ be the number of 2's in the tuple. 
Since  $\sum_{i=0}^{s+d-1} z_i =s$, we have $k \le \lfloor s/2 \rfloor$. Based on the given restrictions, the tuple cannot have two consecutive 2.  Among indices 0 through $s+d-1$, we select all indices $0 \le i_1 \le \cdots \le i_k \le s+d-1$ where $z_{i_j}=2$ for $1 \leq j \leq k$.

Then the ``no adjacent $2$'' condition on $z_j$ is equivalent to the following:
\begin{align*}
&w_1 := i_1 \ge 0, \\
&w_2 := i_2-i_1-1 \ge 1, \\ 
&\vdots   \\
&w_k:= i_k-i_{k-1}-1 \ge 1, \\
&w_{k+1} := s+d-1-i_k \ge 0, \\
&w_1 + w_{k+1} \ge 1.
\end{align*} 

Obviously $w_1+ w_2+\cdots+ w_{k+1} = s+d-k$. If $w_1 = 0$, then $w_{k+1} \ge 1$, so there are ${{s+d-k-1}\choose{k-1}}$ pairs of such $(w_j)$ by Lemma \ref{numintsol}. If $w_1 \ge 1$, then $w_{k+1} \ge 0$, so there are ${{s+d-k}\choose{k}}$ pairs of such $(w_j)$ by Lemma \ref{numintsol}. Once there are $k$ many $2$ in $z_j$, there should be $s-2k$ many $1$ in $z_j$ by the condition  $\sum_{i=0}^{s+d-1} z_i =s$. Therefore the theorem follows. 
\end{proof}

\begin{remark}
There is another result on the number $(a_1, a_2, \cdots, a_n)$-core partitions, where $(a_i)$ forms an arithmetic progression. Fix a positive integer $p$. Let $f_s$ be the number of simultaneous $(s,s+1,…, s+p)$-cores. Xiong \cite{Xionggenerating} gives a recurrence relation satisfied by $f_s$ and gives an expression for the generating function of this sequence.
\end{remark}

Recall $\mathrm{Cat}_{a,b} = \frac{1}{a+b} {{a+b}\choose{a}}$ is the number of $(a,b)$-core partition if $a$ and $b$ are coprime. 

\begin{thm}
\label{abc}
Suppose that $3<a<b$ are coprime, $a | (2b+ c)$, and $c > \frac{1}{2}ab-2b$. Write $m = (c+2b)/a$.  Then, the number of $(a,b,c)$-core partitions is 
$$
\mathrm{Cat}_{a,b} - (m + 1){{b+a-m-3} \choose {a-2}} + {{b+a-m-3}\choose{a-1}}.
$$
\end{thm}

\begin{remark}
Note that all binomial terms in Theorem \ref{abc} is equal to 0 when $c > ab-a-b$, which gives a constant value $\mathrm{Cat}_{a,b}$ for the number of $(a,b,c)$-core partitions. Indeed, for any $c>ab-a-b$, $c$ can be represented as a linear combination of $a$ and $b$, like $c = sa+tb$ for some nonnegative integer $s$ and $t$. Then, an $(a,b,sa+tb)$-core is simply an $(a,b)$-core.
\end{remark}

\begin{proof}[Proof of Theorem \ref{abc}]

The proof is based on the inclusion-exclusion priciple. 
Let 
$$
A_i : = \left\{(z_0, z_1, \cdots z_{a-1}) \in \NN^a : \sum z_j = b, z_i + z_{i+1} > m\right\}.
$$ 
Note first that $ m > b/2$, so $A_i \cap A_{i+2} = \emptyset$ (since $a>3$). Clearly $|A_i|$ and $|A_i \cap A_{i+1}|$ are independent of the choice of $i$. Theorem \ref{mainthm} shows that the number of $(a,b,c)$-core partitions is
\begin{equation}
\label{inexequation}
\frac{1}{a}\left| \left\{(z_0, z_1, \cdots z_{a-1}) \in \NN^a : \sum z_j = b, z_i + z_{i+1} \le m \text{ for all $i$} \right\}\right| .
\end{equation}
By the inclusion-exclusion principle and Lemma \ref{numintsol}, \eqref{inexequation} equals
$$ \frac{1}{a}{{b+a-1} \choose {a-1}} - |A_0| + |A_0 \cap A_1|.$$ 
We now compute $|A_0|$. If $z_0 \ge m+1$, then $z_1 \ge 0$, so there are ${{b+a-m-2}\choose{a-1}}$ many such tuples $(z_0, z_1, \cdots, z_{a-1})$ by Lemma \ref{numintsol}. If $z_0 = l \le m$, then $z_1 \ge m+1-l$, so there are ${{b+a-m-3}\choose{a-2}}$ many such tuples by Lemma \ref{numintsol}. Therefore
$$
|A_0| = {{b+a-m-2}\choose{a-1}} + (m+1){{b+a-m-3}\choose{a-2}}.
$$
Now we compute $|A_0 \cap A_1|$. If $z_1 \ge m+1$, then $z_0, z_2 \ge 0$, so there are  
${{b+a-m-2}\choose{a-1}}$ such tuples by Lemma \ref{numintsol}.  If $z_1 = l \le m$, then $z_0, z_2 \ge m+1-l$, so there are ${{b+a+l-2m-4}\choose{a-2}}$ such tuples by Lemma \ref{numintsol}. It follows that
\begin{align*}
|A_0 \cap A_1| &= {{b+a-m-2}\choose{a-1}} + \sum_{k=a-2}^{b+a-m-4} {{k}\choose{a-2}} \\
&= {{b+a-m-2}\choose{a-1}} + {{b+a-m-3}\choose{a-1}},
\end{align*}
which completes the proof.
\end{proof}

\section*{Acknowledgement}
We are very grateful to Nathan Kaplan for many valuable comments and discussions. We also thank to Aaron Berger and Dennis Eichhorn for comments.

\bibliographystyle{abbrv}
\bibliography{refe}

\end{document}